\documentclass[12pt,reqno]{amsart}
\usepackage{amsmath,amssymb,amscd,latexsym,amsthm,mathrsfs}
\usepackage[unicode]{hyperref}
\usepackage{hypbmsec}
\newtheorem{theorem}{Theorem}
\newtheorem{lemma}{Lemma}

\theoremstyle{remark}
\newtheorem*{remark}{\bf Remark}
\newtheorem*{acknowledgements}{\bf Acknowledgements}
\renewcommand\d{{\mathrm d}}

\newcommand\Li{\operatorname{Li}}
\newcommand\ol{\overline}
\newcommand\wt{\widetilde}
\newcommand\bs{\boldsymbol s}
\let\lf\lfloor
\let\rf\rfloor
\begin{document}

\title{On a family of polynomials related to $\zeta(2,1)=\zeta(3)$}

\author{Wadim Zudilin}
\address{School of Mathematical and Physical Sciences,
University of Newcastle, Calla\-ghan NSW 2308, AUSTRALIA}
\email{wadim.zudilin@newcastle.edu.au}

\date{24 April 2015. \emph{Revised}: 31 March 2017}

\begin{abstract}
We give a new proof of the identity $\zeta(\{2,1\}^l)=\zeta(\{3\}^l)$ of the multiple zeta values,
where $l=1,2,\dots$, using generating functions of the underlying generalized polylogarithms.
In the course of study we arrive at (hypergeometric) polynomials satisfying 3-term recurrence relations,
whose properties we examine and compare with analogous ones of polynomials originated from
an \hbox{(ex-)}\allowbreak conjectural identity of Borwein, Bradley and Broadhurst.
\end{abstract}


\thanks{The work is supported by Australian Research Council grant DP140101186.}

\maketitle

\section{Introduction}
\label{s1}

The first thing one normally starts with, while learning about the multiple zeta values (MZVs)
$$
\zeta(\bs)=\zeta(s_1,s_2,\dots,s_l)
=\sum_{n_1>n_2>\dots>n_l\ge1}\frac1{n_1^{s_1}n_2^{s_2}\dotsb n_l^{s_l}},
$$
is Euler's identity $\zeta(2,1)=\zeta(3)$\,---\,see \cite{BBr} for an account of proofs and generalizations of the remarkable equality.
One such generalization reads
\begin{equation}
\label{id1}
\zeta(\{2,1\}^l)=\zeta(\{3\}^l)
\qquad\text{for}\quad l=1,2,\dots,
\end{equation}
where the notation $\{\bs\}^m$ denotes the multi-index with $m$ consecutive repetitions of the same index $\bs$.
The only known proof of \eqref{id1} available in the literature makes use of the duality relation of MZVs, originally conjectured in~\cite{Ho}
and shortly after established in \cite{Za}. The latter relation is based on a simple iterated-integral representation of MZVs
(see \cite{Za} but also \cite{BBr,BBB,Zu} for details) but, unfortunately, it is not capable of establishing similar-looking identities
\begin{equation}
\label{id1a}
\zeta(\{3,1\}^l)=\frac{2\pi^{4l}}{(4l+2)!}
\qquad\text{for}\quad l=1,2,\dotsc.
\end{equation}
The equalities \eqref{id1a} were proven in~\cite{BBB} using a simple generating series argument.

The principal goal of this note is to give a proof of \eqref{id1} via generating functions and to discuss, in this context,
a related ex-conjecture of the alternating MZVs. An interesting outcome of this approach is a family of (hypergeometric) polynomials
that satisfy a 3-term recurrence relation; a shape of the relation and (experimentally observed) structure of the zeroes of the polynomials
suggest their bi-orthogonality origin \cite{IN,IM,SZ}.

\section{Multiple polylogarithms}
\label{s2}

For $l=1,2,\dots$, consider the generalized polylogarithms
\begin{align*}
\Li_{\{3\}^l}(z)
&=\sum_{n_1>n_2>\dots>n_l\ge1}\frac{z^{n_1}}{n_1^3n_2^3\dotsb n_l^3},
\\
\Li_{\{2,1\}^l}(z)
&=\sum_{n_1>m_1>n_2>m_2>\dots>n_l>m_l\ge1}\frac{z^{n_1}}{n_1^2m_1n_2^2m_2\dotsb n_l^2m_l},
\\
\Li_{\{\ol2,1\}^l}(z)
&=\sum_{n_1>m_1>n_2>m_2>\dots>n_l>m_l\ge1}
\frac{z^{n_1}(-1)^{n_1+n_2+\dots+n_l}}{n_1^2m_1n_2^2m_2\dotsb n_l^2m_l};
\end{align*}
if $l=0$ we set all these functions to be~$1$. Then at $z=1$,
\begin{equation*}
\zeta(\{3\}^l)=\Li_{\{3\}^l}(1)
\quad\text{and}\quad
\zeta(\{2,1\}^l)=\Li_{\{2,1\}^l}(1),
\end{equation*}
and we also get the related alternating MZVs
\begin{equation*}
\zeta(\{\ol2,1\}^l)=\Li_{\{\ol2,1\}^l}(1)
\end{equation*}
from the specialization of the third polylogarithm.

Since
\begin{align*}
\biggl((1-z)\frac{\d}{\d z}\biggr)\biggl(z\frac{\d}{\d z}\biggr)^2\Li_{\{3\}^l}(z)
&=\Li_{\{3\}^{l-1}}(z),
\\
\biggl((1-z)\frac{\d}{\d z}\biggr)^2\biggl(z\frac{\d}{\d z}\biggr)\Li_{\{2,1\}^l}(z)
&=\Li_{\{2,1\}^{l-1}}(z),
\\
\biggl((1+z)\frac{\d}{\d z}\biggr)^2\biggl(z\frac{\d}{\d z}\biggr)\Li_{\{\ol2,1\}^l}(z)
&=\Li_{\{\ol2,1\}^{l-1}}(-z)
\end{align*}
for $l=1,2,\dots$, the generating series
\begin{gather*}
C(z;t)=\sum_{l=0}^\infty\Li_{\{3\}^l}(z)t^{3l},
\\
B(z;t)=\sum_{l=0}^\infty\Li_{\{2,1\}^l}(z)t^{3l}
\quad\text{and}\quad
A(z;t)=\sum_{l=0}^\infty\Li_{\{\ol 2,1\}^l}(z)t^{3l}
\end{gather*}
satisfy linear differential equations. Namely, we have
$$
\biggl(\biggl((1-z)\frac{\d}{\d z}\biggr)\biggl(z\frac{\d}{\d z}\biggr)^2-t^3\biggr)C(z;t)=0,
\quad
\biggl(\biggl((1-z)\frac{\d}{\d z}\biggr)^2\biggl(z\frac{\d}{\d z}\biggr)-t^3\biggr)B(z;t)=0
$$
and
$$
\biggl(\biggl((1-z)\frac{\d}{\d z}\biggr)^2\biggl(z\frac{\d}{\d z}\biggr)
\biggl((1+z)\frac{\d}{\d z}\biggr)^2\biggl(z\frac{\d}{\d z}\biggr)-t^6\biggr)A(z;t)=0,
$$
respectively. The identities \eqref{id1} and identities
\begin{equation*}
\frac1{8^l}\zeta(\{2,1\}^l)=\zeta(\{\ol2,1\}^l)
\qquad\text{for}\quad l=1,2,\dots,
\end{equation*}
conjectured in~\cite{BBB} and confirmed in \cite{Zh} by means of a nice though sophisticated machinery of double shuffle relations and the `distribution' relations
(see also an outline in \cite{BBa}), translate into
\begin{equation*}
C(1;t)=B(1;t)=A(1;2t).
\end{equation*}

Note that
\begin{equation}
\label{id-C}
C(1;t)=\sum_{l=0}^\infty t^{3l}\sum_{n_1>n_2>\dots>n_l\ge1}\frac1{n_1^3n_2^3\dotsb n_l^3}
=\prod_{j=1}^\infty\biggl(1+\frac{t^3}{j^3}\biggr).
\end{equation}
At the same time, the differential equation for $C(z;t)=\sum_{n=0}^\infty C_n(t)z^n$ results in
$$
-n^3C_n+(n+1)^3C_{n+1}=t^3C_n
$$
implying
$$
\frac{C_{n+1}}{C_n}=\frac{n^3+t^3}{(n+1)^3}
=\frac{(n+t)(n+e^{2\pi i/3}t)(n+e^{4\pi i/3}t)}{(n+1)^3}
$$
and leading to the hypergeometric form
\begin{equation}
\label{id-C-hyper}
C(z;t)={}_3F_2\biggl(\begin{matrix} t, \, \omega t, \, \omega^2t \\ 1, \, 1 \end{matrix}\biggm| z\biggr),
\end{equation}
where $\omega=e^{2\pi i/3}$. We recall that
$$
{}_{m+1}F_m\biggl(\begin{matrix} a_0, \, a_1, \, \dots, \, a_m \\ b_1, \, \dots, \, b_m \end{matrix}\biggm| z\biggr)
=\sum_{n=0}^\infty\frac{(a_0)_n(a_1)_n\dotsb(a_m)_n}{n!\,(b_1)_n\dotsb(b_m)_n}\,z^n,
$$
where $(a)_n=\Gamma(a+n)/\Gamma(a)$ denotes the Pochhammer symbol (also known as the `shifted factorial' because $(a)_n=a(a+1)\dotsb(a+n-1)$ for $n=1,2,\dots$).
It is not hard to see that the sequences $A_n(t)$ and $B_n(t)$ from $A(z;t)=\sum_{n=0}^\infty A_n(t)z^n$ and $B(z;t)=\sum_{n=0}^\infty B_n(t)z^n$
do not satisfy 2-term recurrence relations with polynomial coefficients. Thus, no hypergeometric representations of the type \eqref{id-C-hyper}
are available for them.

\section{Special polynomials}
\label{s3}

The differential equation for $B(z;t)$ translates into the 3-term recurrence relation
\begin{equation}
n^3B_n-(n+1)^2(2n+1)B_{n+1}+(n+2)^2(n+1)B_{n+2}=t^3B_n
\label{recB}
\end{equation}
for the coefficients $B_n=B_n(t)$; the initial values are $B_0=1$ and $B_1=0$.

\begin{lemma}
\label{L1}
We have
\begin{equation}
B_n(t)=\frac1{n!}\sum_{k=0}^n\frac{(\omega t)_k(\omega^2t)_k(t)_{n-k}(-t+k)_{n-k}}{k!\,(n-k)!}
=\frac{(t)_n(-t)_n}{n!^2}\,{}_3F_2\biggl(\begin{matrix} -n, \, \omega t, \, \omega^2t \\ -t, \, 1-n-t \end{matrix}\biggm| 1\biggr).
\label{mi}
\end{equation}
\end{lemma}

\begin{proof}
The recursion \eqref{recB} for the sequence in \eqref{mi} follows from application of the
Gosper--Zeilberger algorithm of creative telescoping. The initial values for $n=0$ and $1$ are straightforward.
\end{proof}

\begin{remark}
The hypergeometric form in \eqref{mi} was originally prompted by \cite[Theorem 3.4]{Mi}:
the change of variable $z\mapsto1-z$ in the differential equation for $B(z;t)$ shows that
the function $f(z)=B(1-z;t)$ satisfies the hypergeometric differential equation with upper
parameters $-t$, $-\omega t, \, -\omega^2t$ and lower parameters $0$, $0$.
\end{remark}

It is not transparent from the formula \eqref{mi} (but immediate from the recursion~\eqref{recB})
that $B_n(t)\in t^3\mathbb Q[t^3]$ for $n=0,1,2,\dots$; the classical transformations
of $_3F_2(1)$ and their representations as $_6F_5(-1)$ hypergeometric series (see \cite{Ba})
do not shed a light on this belonging either.

\begin{lemma}
\label{L2}
We have
\begin{equation}
\label{id-B}
B(1;t)=\prod_{j=1}^\infty\biggl(1+\frac{t^3}{j^3}\biggr).
\end{equation}
\end{lemma}

\begin{proof}
This follows from the derivation
\begin{align*}
B(1;t)=\sum_{n=0}^\infty B_n(t)
&=\sum_{n=0}^\infty\frac1{n!}\sum_{k=0}^n\frac{(\omega t)_k(\omega^2t)_k(t)_{n-k}(-t+k)_{n-k}}{k!\,(n-k)!}
\\
&=\sum_{k=0}^\infty\frac{(\omega t)_k(\omega^2t)_k}{k!^2}\sum_{m=0}^\infty\frac{(t)_m(-t+k)_m}{m!\,(k+1)_m}
\displaybreak[2]\\
&=\sum_{k=0}^\infty\frac{(\omega t)_k(\omega^2t)_k}{k!^2}\cdot{}_2F_1\biggl(\begin{matrix} t, \, -t+k \\ k+1 \end{matrix}\biggm| 1\biggr)
\displaybreak[2]\\
&=\frac1{\Gamma(1-t)\Gamma(1+t)}\sum_{k=0}^\infty\frac{(\omega t)_k(\omega^2t)_k}{k!\,(1-t)_k}
\displaybreak[2]\\
&=\frac1{\Gamma(1-t)\Gamma(1+t)}\cdot{}_2F_1\biggl(\begin{matrix} \omega t, \, \omega^2t \\ 1-t \end{matrix}\biggm| 1\biggr)
\displaybreak[2]\\
&=\frac1{\Gamma(1-t)\Gamma(1+t)}\cdot\frac{\Gamma(1-t)}{\Gamma(1-(1+\omega)t)\Gamma(1-(1+\omega^2)t)}
\\
&=\frac1{\Gamma(1+t)\Gamma(1+\omega t)\Gamma(1+\omega^2t)}
=\prod_{j=1}^\infty\biggl(1+\frac{t^3}{j^3}\biggr),
\end{align*}
where we applied twice Gauss's summation \cite[Section 1.3]{Ba}
$$
{}_2F_1\biggl(\begin{matrix} a, \, b \\ c \end{matrix}\biggm| 1\biggr)
=\frac{\Gamma(c)\,\Gamma(c-a-b)}{\Gamma(c-a)\,\Gamma(c-b)}
$$
valid when $\Re(c-a-b)>0$.
\end{proof}

Finally, we deduce from comparing \eqref{id-C} and \eqref{id-B},

\begin{theorem}
The identity $\zeta(\{3\}^l)=\zeta(\{2,1\}^l)$ is valid for $l=1,2,\dots$\,.
\end{theorem}

\section{A general family of polynomials}
\label{s4}

It is not hard to extend Lemma~\ref{L1} to the one-parameter family of polynomials
\begin{align}
B_n^\alpha(t)
&=\frac1{n!}\sum_{k=0}^n\frac{(\omega t)_k(\omega^2t)_k(\alpha+t)_{n-k}(\alpha-t+k)_{n-k}}{k!\,(n-k)!}
\nonumber\\
&=\frac1{n!}\sum_{k=0}^n\frac{(\alpha+\omega t)_k(\alpha+\omega^2t)_k(t)_{n-k}(\alpha-t+k)_{n-k}}{k!\,(n-k)!}.
\label{eqB}
\end{align}

\begin{lemma}
\label{L3}
For each $\alpha\in\mathbb C$, the polynomials \eqref{eqB} satisfy the $3$-term recurrence relation
\begin{equation*}
((n+\alpha)^3-t^3)B_n^\alpha-(n+1)(2n^2+3n(\alpha+1)+\alpha^2+3\alpha+1)B_{n+1}^\alpha+(n+2)^2(n+1)B_{n+2}^\alpha=0
\end{equation*}
and the initial conditions $B_0^\alpha=1$, $B_1^\alpha=\alpha^2$.
In particular, $B_n^\alpha(t)\in\mathbb C[t^3]$ for $n=0,1,2,\dots$\,.
\end{lemma}

In addition, we have $B_n^\alpha\in t^3\mathbb Q[t^3]$ for $\alpha=0,-1,\dots,-n+1$
(in other words, $B_n^\alpha(0)=0$ for these values of~$\alpha$).

\begin{lemma}
\label{L4}
$B_n^{1-n-\alpha}(t)=B_n^\alpha(t)$.
\end{lemma}

\begin{proof}
This follows from the hypergeometric representation
\begin{equation*}
B_n^\alpha(t)=\frac{(\alpha+t)_n(\alpha-t)_n}{n!^2}\,{}_3F_2\biggl(\begin{matrix} -n, \, \omega t, \, \omega^2t \\ \alpha-t, \, 1-\alpha-n-t \end{matrix}\biggm| 1\biggr).
\qedhere
\end{equation*}
\end{proof}

Here is one more property of the polynomials that follows from Euler's transformation \cite[Section 1.2]{Ba}.

\begin{lemma}
\label{L5}
We have
$$
\sum_{n=0}^\infty B_n^\alpha(t)z^n
=(1-z)^{1-2\alpha}\sum_{n=0}^\infty B_n^{1-\alpha}(t)z^n.
$$
\end{lemma}

\begin{proof}
Indeed,
\begin{align*}
\sum_{n=0}^\infty B_n^\alpha(t)z^n
&=\sum_{k=0}^\infty\frac{(\omega t)_k(\omega^2t)_k}{k!^2}\,z^k\cdot{}_2F_1\biggl(\begin{matrix} \alpha+t, \, \alpha-t+k \\ k+1 \end{matrix}\biggm| z\biggr)
\nonumber\displaybreak[2]\\
&=\sum_{k=0}^\infty\frac{(\omega t)_k(\omega^2t)_k}{k!^2}\,z^k\cdot(1-z)^{1-2\alpha}{}_2F_1\biggl(\begin{matrix} 1-\alpha+t, \, 1-\alpha-t+k \\ k+1 \end{matrix}\biggm| z\biggr)
\nonumber\\
&=(1-z)^{1-2\alpha}\sum_{n=0}^\infty B_n^{1-\alpha}(t)z^n.
\qedhere
\end{align*}
\end{proof}

\begin{proof}[Alternative proof of Lemma~\textup{\ref{L2}}]
It follows from Lemma~\ref{L5} that
$$
B_n^1(t)=\sum_{k=0}^nB_k(t),
$$
hence $B(1;t)=\lim_{n\to\infty}B_n^1(t)$ and the latter limit is straightforward from \eqref{eqB}.
\end{proof}

Note that, with the help of the standard transformations of $_3F_2(1)$ hypergeometric series, we can also write \eqref{eqB} as
\begin{equation*}
B_n^\alpha(t)=\frac{(\alpha-\omega t)_n(\alpha-\omega^2t)_n}{n!^2}\,{}_3F_2\biggl(\begin{matrix} -n, \, \alpha+t, \, t \\ \alpha-\omega t, \, \alpha-\omega^2t \end{matrix}\biggm| 1\biggr),
\end{equation*}
so that the generating functions of the continuous dual Hahn polynomials lead to the generating functions
$$
\sum_{n=0}^\infty\frac{n!}{(\alpha-t)_n}\,B_n^\alpha(t)z^n
=(1-z)^{-t}\,{}_2F_1\biggl(\begin{matrix} \alpha+\omega t, \, \alpha+\omega^2t \\ \alpha-t \end{matrix}\biggm| z\biggr)
$$
and
$$
\sum_{n=0}^\infty\frac{(\gamma)_n\,n!}{(\alpha-\omega t)_n(\alpha-\omega^2t)_n}\,B_n^\alpha(t)z^n
=(1-z)^{-\gamma}\,{}_3F_2\biggl(\begin{matrix} \gamma, \, \alpha+t, \, t \\ \alpha-\omega t, \, \alpha-\omega^2t \end{matrix}\biggm| \frac z{z-1} \biggr),
$$
where $\gamma$ is arbitrary.

Finally, numerical verification suggests that for real $\alpha$ the zeroes of $B_n^\alpha$ viewed as polynomials in $x=t^3$
lie on the real half-line $(-\infty,0]$.

\section{Polynomials related to the alternating MZV identity}
\label{s5}

Writing
\begin{align*}
A(z;t)
&=\sum_{n=0}^\infty A_n(t)z^n
\\
&=1+\tfrac14t^3z^2-\tfrac16t^3z^3+\bigl(\tfrac1{192}t^3+\tfrac{11}{96}\bigr)t^3z^4
-\bigl(\tfrac1{240}t^3+\tfrac1{12}\bigr)t^3z^5
\\ &\qquad
+\bigl(\tfrac1{34560}t^6+\tfrac{23}{5760}t^3+\tfrac{137}{2160}\bigr)t^3z^6+O(z^7)
\end{align*}
and using the equation
$$
\biggl((1+z)\frac{\d}{\d z}\biggr)^2\biggl(z\frac{\d}{\d z}\biggr)A(z;t)=t^3A(-z;t),
$$
we deduce that
\begin{align}
(n^3-T)A_n+(n+1)^2(2n+1)A_{n+1}+(n+2)^2(n+1)A_{n+2}&=0,
\label{rec1}
\\ \intertext{where $T=(-1)^nt^3$. Producing two shifted copies of~\eqref{rec1},}
((n-1)^3+T)A_{n-1}+n^2(2n-1)A_n+(n+1)^2nA_{n+1}&=0,
\label{rec2}
\\
((n-2)^3-T)A_{n-2}+(n-1)^2(2n-3)A_{n-1}+n^2(n-1)A_n&=0,
\label{rec3}
\end{align}
then multiplying recursion \eqref{rec1} by $n(n-1)^2(2n-3)$, recursion \eqref{rec2} by $-(n-1)^2\*(2n+\nobreak1)(2n-3)$,
recursion \eqref{rec3} by $(2n+1)((n-1)^3+T)$ and adding the three equations so obtained we arrive at
\begin{align*}
&
(2n+1)((n-1)^3+T)((n-2)^3-T)A_{n-2}
\nonumber\\ &\quad
-n(n-1)(2n-1)(2n(n-1)(n^2-n-1)-3T)A_n
\nonumber\\ &\quad
+(n+2)^2(n+1)n(n-1)^2(2n-3)A_{n+2}
=0.
\end{align*}
This final recursion restricted to the subsequence $A_{2n}$, namely
\begin{align}
&
(4n+5)((2n)^3-t^3)((2n+1)^3+t^3)A_{2n}
\nonumber\\ &\quad
-(4n+3)(2n+1)(2n+2)(2(2n+1)(2n+2)(4n^2+6n+1)-3t^3)A_{2n+2}
\nonumber\\ &\quad
+(4n+1)(2n+1)^2(2n+2)(2n+3)(2n+4)^2A_{2n+4}
=0,
\label{rec}
\end{align}
and, similarly, to $A_{2n+1}$ gives rise to two families of so-called
Frobenius--Stickelberger--Thiele polynomials \cite{STZ}. The latter connection, however, sheds no light on the asymptotics of $A_n(t)\in\mathbb Q[t^3]$.
Unlike the case of $B(z;t)$ treated in Section \ref{s3} we cannot find closed form expressions for those subsequences.
Here is the case most visually related to the recursion \eqref{rec}:
\begin{align}
&
(4n+5)\frac{(2n)^3-t^3}{t-2n}\,\frac{(2n+1)^3+t^3}{t+2n+1}\,A_n'
\nonumber\\ &\quad
-(4n+3)(2n+1)(2n+2)((2n+1)(2n+2)+(8n^2+12n+1)t+3t^2)A_{n+1}'
\nonumber\\ &\quad
+(4n+1)(2n+1)^2(2n+2)(2n+3)(2n+4)^2A_{n+2}'
=0,
\nonumber
\end{align}
where
$$
A_n'=\frac1{2^n\,(1/2)_n\,n!}\sum_{k=0}^n\frac{(\omega t/2)_k(\omega^2t/2)_k(t/2)_{n-k}(1/2)_{n-k}}{k!\,(n-k)!}\,(-1)^k.
$$
The latter polynomials are not from $\mathbb Q[t^3]$.

If we consider $\wt A_n(t)=\sum_{k=0}^nA_k(t)$ then (it is already known \cite{BBBL,Zh} that)
$$
(n^3-(-1)^nt^3)\wt A_{n-1}+(2n+1)n\wt A_n-(n+1)^2n\wt A_{n+1}=0, \qquad n=1,2,\dotsc.
$$
As before, the standard elimination translates it into
\begin{align}
&
(2n+3)((n-1)^3+T)(n^3-T)\wt A_{n-2}
\nonumber\\ &\quad
-(2n+1)n(n-1)(2(n^2+n+1)^2-6-T)\wt A_n
\nonumber\\ &\quad
+(2n-1)(n+2)^2(n+1)^2n(n-1)\wt A_{n+2}
=0,
\nonumber
\end{align}
where $T=(-1)^nt^3$. One can easily verify that
$$
\wt A_n(t)=1+\dots+\frac{t^{3\lf n/2\rf}}{2^{\lf n/2\rf}\lf n/2\rf!\,n!}
$$
but we also lack an explicit representation for them.

We have checked numerically a fine behaviour (orthogonal-polynomial-like) of the zeroes of $A_n$ and $\wt A_n$
viewed as polynomials in $x=t^3$ (both of degree $[n/2]$ in~$x$). Namely, all the zeroes lie on the real half-line $(-\infty,0]$.
This is in line with the property of the polynomials $B_n^\alpha$ (see the last paragraph in Section~\ref{s4}).

\begin{acknowledgements}
The work originated from discussions during
the research trimester on Multiple Zeta Values, Multiple Polylogarithms and Quantum Field Theory at ICMAT in Madrid (September--October 2014)
and was completed during the author's visit in the Max Planck Institute for Mathematics in Bonn (March--April 2015);
I thank the staff of the institutes for the wonderful working conditions experienced during these visits.
I am grateful to Valent Galliano, Erik Koelink, Tom Koornwinder and Slava Spiridonov for valuable comments
on earlier versions of the note. I am thankful as well to the two anonymous referees for the valuable feedback on the submitted version.
\end{acknowledgements}


\end{document}